\documentclass{amsart}

\usepackage{amsmath, amsthm, amscd, amsfonts, amssymb, graphicx, color,float,pgf,tikz,bm}
\usepackage[english]{babel}
\usepackage{amssymb,fontenc}
\usepackage{cite}
\usepackage{latexsym,wasysym,mathrsfs}
\usepackage{hyperref}
\usetikzlibrary{arrows}

\newtheorem{theorem}{Theorem}[section]
\newtheorem{lemma}[theorem]{Lemma}
\newtheorem{proposition}[theorem]{Proposition}
\newtheorem{corollary}[theorem]{Corollary}

\theoremstyle{definition}
\newtheorem{definition}[theorem]{Definition}

\theoremstyle{remark}

\newtheorem{remark}[theorem]{Remark}
\newtheorem{hyp}[]{Property}
\newtheorem{con}[]{Condition}
\numberwithin{equation}{section}
\newcommand{\R}{{\mathbb R}}

\newcommand{\SF}{{\mathbb S}}

\newcommand{\loc}{{\rm loc}}

\DeclareMathOperator*{\es}{esssup}
\DeclareMathOperator*{\esi}{essinf}

\def\B{{\mathcal B}}

\def\cK{{\mathcal K}}
\def\cM{{\mathcal M}}

\def\cR{{\mathcal R}}

\def\cS{{\mathcal S}}
\newcommand{\cL}{{\mathcal L}}
\def\Id{{\mathcal Id}}
\def\cR{{\mathcal R}}
\def\cT{{\mathcal T}}

\def\loc{{\rm loc}}

\def\diam{{\rm diam\,}}

\numberwithin{equation}{section}

\title[Elliptic equations in weighted Banach spaces]{Higher order elliptic equations  in weighted Banach  spaces}
\author[B.T. Bilalov]{Bilal T. Bilalov}
\address{Institute of Mathematics and Mechanics\\
 NAS of Azerbaijan, Baku, Azerbaijan\\
 e-mail:  b\_bilalov@mail.ru}

\author[S.R. Sadigova]{ Sabina R. Sadigova}
\address{Institute of Mathematics and Mechanics\\
 NAS of Azerbaijan, Baku, Azerbaijan\\
 e-mail: s\_sadigova@mail.ru}

\author[L.G. Softova]{Lyoubomira G. Softova}
\address{Department of Mathematics
University of Salerno, Italy\\
 e-mail:  lsoftova@unisa.it}



\begin{document}

\begin{abstract}
We consider $m$-th order linear, uniformly elliptic equations  $\cL u=f$  with non-smooth coefficients in Banach-Sobolev spaces $W_{X_w}^m (\Omega )$ generated by weighted   Banach Function Spaces  (BFS) $X_w (\Omega)$  on a bounded domain $\Omega \subset \R^{n}.$ 
 Supposing boundedness of the Hardy-Littlewood  Maximal operator  and the Calder\'on-Zygmund singular integrals in $X_w (\Omega)$ we obtain solvability in the small in $W_{X_w}^m (\Omega )$ and establish interior  Schauder type a priori estimates. 
These results will be used in order to obtain Fredholmness of the operator  $\cL $  in    $X_w (\Omega).$
\end{abstract}



\maketitle

\section{Introduction}

The question of solvability in the small and Schauder type estimates play an essential role in the theory of the Fredholmness of boundary value problems for elliptic equations in appropriate BFS. There is a vast number of papers and monographs dedicated to the Fredholmness of  Partial Differential Equations (PDEs)  in the frame of  classical function classes  as the H\"older and Sobolev spaces. The appearance of new function spaces and their intensive study had a very big impact on the regularity and existence theory for  PDEs.   The so-called {non standard function spaces} turn to be interesting not only from theoretical point of view but also for the applied mathematics and mathematical physics. Among the most studied spaces we can find the Morrey $L^{p,\lambda } $ spaces, the weighted Lebesgue spaces and various  generalizations,   the variable Lebesgue spaces $L^{p(\cdot)}, $  the grand Lebesgue spaces $L^{p)} ,$   the Orlicz spaces $L_{\Phi } $ and many others (see \cite{APS,8,9,34,GS,28,21,24,25,26} and the references therein).

 In the present paper we are interested on  elliptic differential operators of higher order in weighted general BFS $X_w (\Omega)$ with   Muckenhoupt type weight $w$  in a bounded domain $\Omega \subset \R^{n}$ with a smooth enough boundary.   In order to obtain the desired  a priori estimates  we need that the Hardy-Littlewood maximal and the Calder\'on-Zygmund singular operators are bounded in $X_w (\Omega).$  Then under certain conditions on the coefficients we obtain solvability in the small, that is local solvability,  in Sobolev spaces $W_{X_w}^m (\Omega )$ generated by weighted BFS $X_w (\Omega).$

 More over, we obtain interior Schauder type inequalities in $X_w (\Omega)$ for the solutions of the linear uniformly elliptic equations under consideration.  Such estimates play an exceptional role in the establishing of  the Fredholmness of the corresponding elliptic operators. At the end, some examples of BFS are given. 

 In what follows we use the standard notation:
\begin{itemize}

%\item
  % $x=(x_1,\ldots,x_n )\in \R^{n} $ with $|x|=\big(\sum_{i=1}^{n}x_{i}^{2} \big)^{1/2 }, $
   %$\B_{r} (x)=\{y\in \R^{n} :|x-y|<r\}$,  $\B_{r} \equiv \B_{r}(0)$ ;

\item 
  $\Omega \subset \R^{n} $ is a bounded domain with Lebesgue measure  $|\Omega |,$
  $d_{\Omega } =\diam  \Omega ,$ \newline
$\B_{r} (x)=\{y\in \R^{n} :|x-y|<r\}$,  $\B_{r} \equiv \B_{r}(0),$ \newline
  $\Omega _{r} (x_{0} )=\Omega \cap \B_{r} (x_{0}),$ \ 
 $\Omega _{r} =\Omega _{r} (0);$
\item
  $[X;Y]$ is the  Banach space of bounded operators acting from $X$ to  $Y$ endowed by the operator norm 
  $\|  \cdot\| _{[X;Y]} $ , $[X]=[X;X];$
%\item
%  $\Id$ is the identity  operator in  the corresponding space;
 
\item
 $\chi_{E} $ stands for the characteristic function of the set   $E\subset \R^{n} $; 
\item
 $\alpha =(\alpha_{1},\ldots,\alpha _{n} ),$  is a  multi-index, $\alpha_i\geq 0$ and  $|\alpha |=\sum_{i=1}^{n}\alpha_{i};  $
 $$
D_{x_i}u=D_i u=\partial u/ \partial x_{i}, \quad    D^{\alpha } u=D^{\alpha_1}_{x_1}D^{\alpha_2}_{x_2}\ldots D^{\alpha_n }_{x_n} u; 
$$
\item
  for any $u:\Omega \to \R$ denote  $D^{k} u= \sum_{|\alpha|=k} D^\alpha u;$  
\item
For any normed space $X$ we accept the notion 
$$
\| D^{k} u\|_{X(\Omega )} =\sum_{|\alpha|=k}\| D^{\alpha } u\|_{X(\Omega )};  
$$
%\item
%  $X\hookrightarrow Y$ is a continuous embedding of  $X$ into  $Y;$

%\item
%$C_{0}^{\infty}(\Omega)$ is the set of infinitely differentiable functions with compact support in $\Omega$;
%\item $C^{m}(\Omega)$ is the set of $m$-th order continuous differentiable functions in $\Omega$;
\item $F(\Omega)$ is the set of Lebesgue measurable functions in $\Omega;$
%\item
%$W^{p,m}(\Omega),$ $p>1, m=1,2,\ldots$ is the classical  Sobolev space. 
%\item
%$\overset{\circ }{W}{}^{p,m}(\Omega)$   is the closure of  $C_{0}^{\infty}(\Omega)$  with respect to the norm in ${W}^{p,m}(\Omega)$;
\item $\overline{\overline{o}}(1)$ stands for a quantity that tends to 0  as $r\to0;$

%\item
%The letter $c$ denotes a positive constant, depending on known quantities. Its value may  vary   from one occurrence to another.

\end{itemize}

\section{Definitions and  auxiliary results}

The question of existence and regularity of the solutions of linear elliptic equations is strongly related to the question of continuity of certain integral operators in the corresponding function spaces. For this goal we  recall the definitions of these operators and introduce the  function spaces that we are going to use.

Let $f\in L^1(\R^n)$ and $\cM f$ be  the  Hardy-Littlewood maximal operator 
\begin{equation}\label{eq-M}
\cM  f(x)=\sup_{r>0} \frac{1}{|\B_{r} (x)| }\int_{\B_{r}(x)}|f(y)|\,dy,
\end{equation}
and $\cK f$ be the Calder\'on-Zygmund integral operator
\begin{equation}\label{eq-S}
\cK  f(x)=p.v. \int _{\R^{n} }\frac{\omega(y)}{|x-y|^{n } } f(y) \,  dy.
\end{equation}
For each $\gamma \in(0,n)$  we consider  the Riesz potential  
\begin{equation}\label{eq-R}
\cR_{\gamma } f(x)=\int _{\R^{n} }\frac{f(y)}{|x-y|^{n-\gamma } } \,  dy.
\end{equation}
Simple calculations show that  we can  estimate the Riesz potential via the maximal function  (see \cite{40})
\begin{equation}\label{Lemma 2.1}
\int _{|y-x|\le \delta }\frac{|f(y)|\, dy}{|y-x|^{n-\gamma } }  \le c \delta ^{\gamma } \cM  f(x), \qquad \forall \, x\in \R^n.
\end{equation}

Following  \cite{2}, we  define the   BFS   and give some of its properties.
\begin{definition}\label{def-BFS} 
Let $F(\R^n)$ be the set of Lebesgue measurable functions on $\R^n.$  A mapping $\|\cdot\|_X:F(\R^n)\to [0,\infty]$ is called  {\it  Banach function norm} if for all $f,g,f_n\in F(\R^n),$ $n=1,2,\ldots,$ for any constant  $a\in \R$ and for any Lebesgue measurable set $E\subset\R^n$ the following properties hold:
\begin{enumerate}
\item[(P1)]  $\|f\|_X=0 \  \Longleftrightarrow \  f=0 $ \quad a.e. in \   $\R^n,$\newline
$\|af\|_X=|a|\|f\|_X,$   \quad
   $\|f+g\|_X\leq \|f\|_X+\|g\|_X;$

\item[(P2)] if $0\leq g\leq f $ \  then \  $\|g\|_X\leq \|f\|_X;$
\item[(P3)] if $0\leq f_n \uparrow f$ \ a.e. in \  $\R^n,$ \  then \  $\|f_n\|_X\uparrow \|f\|_X;$
\item[(P4)] if $|E|<\infty$ \ then \  $\|\chi_E\|_X<\infty;$
\item[(P5)]  if $|E|<\infty $ \  then \  $\int_E |f(x)|\, dx\leq  c  \|f\|_X$, \  with a  constant independent  of  $f.$
\end{enumerate} 
The  collection of all functions $f\in F(\R^n)$ for which $\|f\|_X<\infty$ is called Banach Function Space $X(\R^n).$
\end{definition}

We can extend the Definition~\ref{def-BFS} to  BFSs  {$X(\Omega)$} builded   on  a bounded domain $\Omega$ with $|\Omega|>0$ and { regular boundary $\partial \Omega$} by  taking $f\in F(\Omega)$ and extending it as zero out of $\Omega.$

It  follows immediately, by $(P5)$  that   $X(\Omega)\subset L^1(\Omega).$  
The associated space $X'(\Omega),$ endowed  with  the associated norm $\|\cdot\|_{X'},$  consists of all  $g\in F(\Omega),$ such that  (cf. \cite{2})
$$
\|g\|_{X'(\Omega)}=\sup\Big\{ \int_{\Omega} |f(x)g(x)|\,dx : \  \|f\|_{X(\Omega)}\leq 1  \Big\}<\infty.
$$  
At this point we can  extend  the H\"older inequality in the  case of BFSs.
\begin{proposition} Let $X(\Omega)$ and $X'(\Omega)$ be associated BFSs. If $f\in X(\Omega)$ and  $g\in X'(\Omega),$ then $fg\in L^1(\Omega)$  and 
\begin{equation}\label{eq-Holder} 
\int_\Omega |f(x)g(x)|\,dx\leq \|f\|_{X(\Omega)}\|g\|_{X'(\Omega)}\,.
\end{equation}
\end{proposition}
Let $w$ be a {\it positive weight, that is a Lebesgue measurable  function on $\Omega$ for which $0<w(x)<\infty.$}   Then we can define $X_w(\Omega)$  as  the space of all  $f\in F(\Omega)$ for which  $fw\in X(\Omega),$ that is 
$$
\|f\|_{X_w(\Omega)}= \|fw\|_{X(\Omega)}<\infty.
$$ 
In what follows we assume that $w\in X(\Omega)$  and $w^{-1}\in X'(\Omega).$

The definition of general Sobolev  BFS follows naturally from the definition of the classical Sobolev spaces.  By 
 $W_{X_w }^m (\Omega)$ we  denote the  Sobolev space of all  functions $ f\in X_w(\Omega)$ differentiable in  distributional sense up to order $m$,  that is
$$
W_{X_w}^m (\Omega)=\{ D^\alpha f\in X_w (\Omega),  \quad    \forall \, \alpha  : 0\leq |\alpha|\leq m \},
$$
endowed by the norm 
\begin{equation}\label{norm1}
\| f\|_{W_{X_w}^m (\Omega )} =\sum_{|\alpha|\le m}\| D^\alpha f\| _{X_w (\Omega )}.
\end{equation}

{We denote by  
$\overset{\circ}{W}{}^{m}_{X_w}(\Omega)$  the space of $W^{m}_{X_w}(\Omega)$-functions compactly supported in $\Omega,$  endowed   with  the same norm \eqref{norm1}.}

In the  case when $\Omega \equiv \B_{r}(0) $  we  simplify the notion, writing  $X_w(r)$ and  $W_{X_w}^m (r)$ instead of  $X_w(\B_r)$ and  $W_{X_w}^m (\B_r).$   

 In our further considerations, we are going to use the  following norm
\begin{equation}\label{norm2}
\| f\|_{W_{X_w ;d_\Omega }^{m} (\Omega)} =\sum_{|\alpha|\le m}d_{\Omega }^{|\alpha|}\| D^{\alpha} f\| _{X_w (\Omega)},  
\end{equation} 
which is   equivalent of \eqref{norm1}. 

To be able to adapt  the  classical techniques from the $L^p$-theory  to nonstandard function spaces we assume that that space $X_w(\Omega)$ and the weight function 
$w$  possess the following properties. 

\begin{hyp}\label{Hyp-2-1}  Let   $X_w(\Omega )$  be a weighted BFS,  suppose that  $w$  is such that the following inclusions hold:
\begin{enumerate}
\item[($A$)] The operators \eqref{eq-M} and \eqref{eq-S} are bounded in $X_w(\Omega),$ i.e.  $ \cM, \cK \in [X_w (\Omega )]$ and the estimates hold
\begin{equation}\label{eq-operators}
\|\cM f\|_{X_w(\Omega)}\leq c \|f\|_{X_w(\Omega)}, \quad  \| \cK f\|_{X_w (\Omega )}\leq c \|f\|_{X_w(\Omega)}
\end{equation}
with constants independent of $f.$

\item [($B$)] There exists $p_0 \in (1,+ \infty)$ such that 
\begin{equation}\label{eq-spaces}
 X_w(\Omega) \subset L^{p}(\Omega),  \qquad \forall \ p\in[1,p_0].
\end{equation}

\end{enumerate}
\end{hyp}

\begin{hyp}\label{Hyp-2-2} 
For any bounded domain $\Omega$  with $\partial\Omega\in C^m,$  the   Sobolev-Banach  space $W_{X_w }^m (\Omega )$ has  the {\it extension property}. This means that   for each domain $\Omega'$ such that $\Omega \Subset  \Omega', $ there exists a linear bounded  operator, called   {\it extension operator}, such that 
\begin{equation}\tag{E}\label{E}
\begin{cases}
&\theta :\, W_{X_w }^{k} (\Omega)\to W_{X_w }^{k} (\Omega'), \qquad   \theta f \big\vert_{\Omega} =f,\\[6pt]
&\| \theta  f\|_{W_{X_w }^{k} (\Omega' )} \le c\| f\|_{W_{X_w }^{k} (\Omega)},\qquad   \forall\,  k=0,\ldots,m.
\end{cases}
\end{equation}
\end{hyp}

\begin{hyp}\label{Hyp-2-3}
{For a given space $X_w(\Omega)$ we suppose that  one of the following conditions hold.}
\begin{con}\label{Con-2-1}
There exist $p_{1}\in (1, \infty)$ and a constant $c>0$, such that for all disjoint  partition of $\Omega$: $\Omega=\bigcup_k \Omega_k,$  \mbox{$\Omega_k\cap \Omega_j=\emptyset ,$}  if  $ k\not= j,$  and for all $1\leq p\leq p_1,$ it holds
$$
\| f\|_{X(\Omega)} \le c\Big( \sum_{k}\| f\|_{X(\Omega_{k} )}^{p }\Big)^{1/p } \qquad  \forall\,  f\in X(\Omega ).
$$
\end{con}
\begin{con}\label{Con-2-2}
There exists $ p_2 \in (1, \infty)$ such that   
$$
L^p(\Omega)\subset  X_w (\Omega), \qquad \forall \  p\in[p_2,\infty).
$$
\end{con}
\end{hyp}

 \begin{remark}\label{Remark 2.1.} 
{Let us note that the Condition~\ref{Con-2-1} is verified if $X(\Omega)\equiv L^p(\Omega).$ }

{ Concerning the Condition~\ref{Con-2-2} we can construct the following example.  Let $w$ be a positive weight belonging to    $L^{1+0}(\Omega)=\cup_{\delta>0}L^{1+\delta}(\Omega),$ and consider the space $X_w(\Omega)\equiv L^{p(\cdot)}_w(\Omega)$ with $p(\cdot)$ measurable, verifying 
\begin{equation}\label{eq-p}
1<p_+\leq p_+<\infty, \qquad\quad
p_{-} =\esi_{\Omega }  p(x), \quad 
p_{+} =\es_{\Omega } p(x).
\end{equation}
There exists $\alpha>1$ such that $w\in L^\alpha(\Omega)$ and if we denote $p_1(\cdot)=\alpha'p(\cdot),$ then 
$$
p_1^-=\alpha' p^-, \qquad  p_1^+ =\alpha' p^+
$$ 
and the following continuous embeddings hold
\begin{equation}\label{eq-p1}
L^{p_1^+}(\Omega) \hookrightarrow  L^{p_1(\cdot)}(\Omega) \hookrightarrow L^{p_1^-}(\Omega).
\end{equation}
Applying the H\"older inequality we obtain
$$
\int_\Omega |f(x)|^{p(x)}w(x)\, dx\leq  \|w\|_{L^\alpha(\Omega)}    \left(   \int_\Omega |f(x)|^{p_1(x)}\,dx \right)^{\frac1{\alpha'}}
$$
that implies $L^{p_1^+}(\Omega)\hookrightarrow  L^{p(\cdot)}_w(\Omega)$ and hence 
$$
L^p(\Omega) \hookrightarrow L^{p(\cdot)}_w(\Omega)\qquad \forall\,  p\in [ p_1^+,\infty).
$$}

In addition, the classical weighted Lebesgue spaces  $L^p_w(\Omega), p\in(1,\infty), w\in A_p(\Omega)$  verify  the Properties~\ref{Hyp-2-1},{}~\ref{Hyp-2-2}, and~\ref{Hyp-2-3} (cf. \cite{14} .

It is easy to see that the  \eqref{E} holds for any weighted space $X_w(\Omega), $ just extending the  functions  as zero out of $\Omega,$ while 
 for the Sobolev spaces it is not so obvious. This property is well-known in the  case of classical Sobolev spaces $W^{p,m} (\Omega ), p\in(1,+\infty)$ (see \cite{3})  and weighted  Sobolev spaces {$W^{p,m}_w(\Omega)$} with a Muckenhoupt weight (see \cite{14,CF,17}). In the our case,  we assert  that the extension property holds  for all  $k=0,\ldots,m,$

Concerning the Conditions~\ref{Con-2-1} and~\ref{Con-2-2} they are not mutually exclusive and we suppose that for a given BFS at least one of them holds true.
\end{remark}

%\section{Main lemma}

\section{Statement of the problem}

We consider the $m$-order linear differential operator 
\begin{equation} \label{eq-L} 
\cL(x,D)=\sum _{|\alpha|\le m}a_{\alpha} (x)  D^{\alpha} ,      
\end{equation}
with   $m$ being  {\it even number.}  

The operator  $\cL(x,D)$ is {\it  uniformly  elliptic} that is, 
 there exist  positive constants $\lambda$ and $\Lambda,$ such that 
\begin{equation}\label{el-cond}
\lambda |\xi|^{m}\leq \sum_{|\alpha|=m}a_\alpha(x) \xi^\alpha\leq \Lambda |\xi|^{m},\quad \text{for a.a. } x\in\Omega, \  \forall\, \xi\in\R^n.
\end{equation}

We say that the operator $\cL$ satisfies in $x_0\in \Omega$ the \eqref{eq-P}-property  if there exists a ball $\B_r(x_0)\Subset \Omega $ and functions $ g_\alpha \in L^\infty(\B_r(x_0)) ,$  $ |\alpha|=m,$ continuous in $x_0,$ such that 
\begin{equation}\tag{$P_{x_0}$}\label{eq-P}
\begin{cases}
 a_\alpha(x)=  g_\alpha(x) & \text{for a.a. } x\in \B_{r} (x_{0} )\setminus \{x_0\},\\
\|g_\alpha -g_\alpha(x_0)\|_{ L^\infty(\B_r(x_0))}\to 0 &\text{as  } r \to 0,  \ \forall \, |\alpha|=m,\\ 
a_{\alpha} \in L^{\infty } (\B_{r} (x_{0})) & \forall \, \alpha: \  |\alpha|< m.
\end{cases}
\end{equation}

Let   $x_0\in \Omega$ be a point in  which \eqref{el-cond} and   \eqref{eq-P}  hold. Consider  the  {\it tangential operator}  
\begin{equation} \label{eq-Lx0} 
L_{x_0} =\sum_{|\alpha|=m}a_{\alpha}(x_{0} )D^{\alpha} 
\end{equation} 
where
under the value $a_\alpha(x_0)$ we understand the value of $g_\alpha(x_0)$ defined by $(P_{x_0})$.
 The {\it fundamental solution} $J_{x_{0} }$ of the equation   $L_{x_{0} } \varphi =0$  is called a {\it parametrix} for the equation $\cL\varphi=0,$ having  a singularity at the point $x_{0}. $ 
By the properties of the fundamental solution, for any multiindex $\beta$ we have
\begin{equation}\label{eq2a}
\begin{split}
&i) \   D^\beta J_{x_0}(x)\in C^\infty(\R^n\setminus\{0\}), \quad \forall \  \beta,\\
&ii) \  D^\beta J_{x_0}(\mu x)=\mu^{-n} D^\beta J_0(x),	\quad \forall \ \beta: |\beta|=m,\\
&iii) \   \int_{\SF^{n-1}} D^\beta J_{x_0}(\xi)\, d\sigma_\xi=0, \quad \forall  \  \beta: |\beta|=m,\\
&iv)  \   |D^{\beta} J_{x_0}(x)|\leq C|x|^{m-n-|\beta|}, \qquad  |\beta|=0,1,2,\ldots, m.
\end{split}
\end{equation}
These properties ensure  that the $m$-order derivatives of  $J_{x_0}$ are   {\it Calder\'on-Zygmund kernels } {(see \cite{23,24} and the references therein for more details). It is well know, 
from the classical theory,   that for any  function $\varphi\in C^m_0(\Omega)$  the following representation holds}
\begin{equation}\label{eq-phi}
\varphi(x)=\int_\Omega J_{x_0}(x-y)(L_{x_0}-\cL)\varphi(y)\, dy + \int_\Omega J_{x_0}(x-y)\cL\varphi(y)\, dy.
\end{equation}
Introducing the operators 
\begin{equation}\label{eq-Sx0}
\cS_{x_{0} }\varphi(x)=\int_\Omega  J_{x_{0} } (x-y) \varphi (y)\,dy
\end{equation}
\begin{equation}
\label{eq-Tx0} 
\begin{split}
\cT_{x_{0} } \varphi(x)&= \cS_{x_{0} } (L_{x_{0} } -\cL)\varphi(x)\\
&=     \int_\Omega J_{x_0}(x-y) \sum_{|\alpha|=m}(a_\alpha(x_0) -a_\alpha(y))D^\alpha\varphi(y)\, dy\\
& -\int_\Omega J_{x_0}(x-y)\sum_{|\alpha|<m} a_\alpha(y)D^\alpha\varphi(y)\, dy
\end{split}
\end{equation}
we can rewrite \eqref{eq-phi} in the form
$$
\varphi(x)=\cT_{x_0}\varphi(x) +\cS_{x_0} \cL\varphi(x).
$$
Our goal is to show that $\cT_{x_0}+\cS_{x_0} \cL$ is identity operator in $W^m_{X_w}(\Omega).$ 
Let us calculate  the $m$-order derivatives of  $\cS_{x_0}\varphi$  (cf. \cite{3,18,26}) 
\begin{equation}\label{eq-2b}
D^\beta\cS_{x_{0} }\varphi(x)=\int_\Omega D^\beta_x J_{x_{0} } (x-y) \varphi (y)\,dy + c\varphi(x)=\cK_\beta \varphi(x) + c\varphi(x)
\end{equation}
with $|\beta|=m.$
Hence  $\cK_\beta \varphi $ are  Calder\'on-Zygmund integrals    satisfying \eqref{eq-operators}. Calculating the   lower order derivatives we observe that they have weak singularity and can be treated as Riesz potential  \eqref{eq-R}.

\begin{lemma}[Main Lemma]\label{MainLemma}  Let the Property~\ref{Hyp-2-1}  and condition \eqref{eq-P}  hold at some point $x_{0}\in\Omega. $
 If  $\varphi \in W_{X_w ;r}^{m} (\B_{r} (x_{0}))$ has a compact support in $\B_r(x_0)\Subset\Omega,$  then   
$$
\| \cT_{x_{0} } \varphi \|_{W_{X_w ;r}^{m} (\B_{r} (x_{0} ))} \le \sigma (r)\| \varphi \|_{W_{X_w ;r}^{m} (\B_{r} (x_{0} ))} ,
$$ 
where $\sigma(r)\to 0$ as $ r\to 0$ and it  depends on the coefficients of $\cL$, but not  on $\varphi. $ 
\end{lemma}
\begin{proof}
Following \cite{8}  we assume that $x_0 =0$ and simplify the notation writing   $\cS_0 $, $L_0 $ and $\cT_0 .$ 
 Let $n\ge 3$ be an  odd number, in case 
of even dimension   it can be  introduced  a fictitious
new variable and extend all functions as constants along the new variable.  
  Take  an arbitrary  $\varphi \in W_{X_w}^m (r)$   with a compact support in $\B_r,$ then
\begin{equation}\label{eq-psi}
\begin{split}
 (L_{0} -\cL) \varphi(x)&= \sum_{|\alpha|=m} (a_\alpha(0)-a_\alpha(x)) D^\alpha \varphi(x)\\
& -\sum_{|\alpha|<m} a_\alpha(x)D^\alpha\varphi(x)
 =:\psi_1 (x) -\psi_2 (x).
\end{split}
\end{equation}
Hence
\begin{equation}\label{eq-Tx0a}
\begin{split}
\cT_0\varphi(x)=& \cS_0\psi_1(x)+\cS_0\psi_2(x)\\
=& \sum_{|\alpha|=m}\int_\Omega J_0 (x-y)  \big(a_\alpha(0) -a_\alpha(y)\big)D^\alpha\varphi(y) \, dy\\
-&  \sum_{|\alpha|<m} \int_\Omega J_0 (x-y)  a_\alpha(y)D^\alpha\varphi(y) \, dy.
\end{split}
\end{equation}
 Since  $\|a_{\alpha}(0)-a_\alpha\|_{L^\infty(r)}\to 0$ as $ r\to 0,$ we have
\begin{equation}\label{eq-psi1}
\begin{split}
\| \psi _{1} \|_{X_w (r)}< \overline{\overline{o}}(1) \| D^m \varphi \|_{X_w (r)},\quad 
\|\psi_2\|_{X_w(r)} <  c \sum_{k<m}\|D^k\varphi\|_{X_w(r)}.
\end{split}
\end{equation}

In order to estimate the Sobolev-Banach norm of \eqref{eq-Tx0a}  we need to calculate the derivatives up to order $m.$ Let $\beta$ be a multi-index and  for   
$|\beta|<m$  we have 
$$ 
 D^\beta  \cT_0\varphi(x)=\int_{\B_{r} } D_{x}^\beta J_{0} (x-y) \psi(y)\,dy .
$$ 
By \eqref{eq2a}    we have 
$$
|D^\beta \cT_0\varphi(x)|\leq c \int_{\B_r }|x-y|^{m-n-|\beta|}  |\psi (y)|\,dy=c \cR_\gamma| \psi(x)|,
$$
where $\cR_\gamma$ is  defined by     \eqref{eq-R}   with $\gamma =m-|\beta|.$ 
Applying \eqref{Lemma 2.1}  and  Property~\eqref{Hyp-2-1} 
we obtain
$$
\| D^\beta  \cT_0\varphi\| _{X_w  (r)} \leq  c  r^\gamma  \| \cM \psi \|_{X_w (r)} \leq c r^\gamma \|\psi \|_{X_w(r)}.
$$
Making use of  \eqref{eq-psi1} we get
\begin{equation} \label{eq-3-1}
\begin{split}
r^{|\beta|}& \| D^\beta \cT_0\phi  \| _{X_w (r)} \leq c r^{m} \| \psi \| _{X_w (r)}
\leq c  r^{m} \big( \| \psi _1 \|_{X_w (r)} +\| \psi_2 \|_{X_w (r)} \big)\\ 
&\leq  
c \Big(\overline{\overline{o}}(1)  r^{m} \| D^m \varphi \|_{X_w (r)}  +\sum_{k\leq m-1} r^{m-k}  r^k \| 
D^k \varphi \| _{X_w (r)}  \Big)\\  
&\leq
 c \Big( \overline{\overline{o}}(1)r^{m} \| D^m \varphi\|_{X_w (r)}  +r\sum_{k\leq m-1} r^k \| D^k \varphi \|_{X_w (r)}  \Big)\\
& \leq \sigma(r) \|\varphi\|_{W_{X_w ;r}^{m}(r)},
\end{split}
\end{equation}
for $r$ small enough and $\sigma(r)$ vanishing function  as $r\to 0.$

Consider now the case $|\beta|=m$. By 
\eqref{eq-spaces} it follows that $W_{X_w}^m (r)\subset W^{1,m} (r )$ and therefore,  it holds (cf. \cite{3})
\begin{equation} \label{eq-3-2} 
D^{\beta} \cT_0\varphi (x)=\int_{\B_r } D^\beta J_{0} (x-y)\psi (y)\, dy+c\psi (x)  \quad \text{ for a.a. } \     x\in\B_{r}                        
\end{equation} 
 where $c$ depends on known quantities but not on  $r$ and $\psi. $
 By the properties of  the kernel it follows that  $D^\beta J_{0} $  is   singular for each $|\beta|=m$ and therefore  by \eqref{eq-operators} and  from \eqref{eq-3-2} we have 
$$
\| D^\beta \cT_0\varphi \|_{X_w(r)} \leq c\| \psi \|_{X_w(r) } .
$$
Hence the following estimate  holds 
\begin{align*}
r^m  \| D^\beta \cT_0\varphi &\|_{X_w (r)} \leq  c r^m \|  \psi \|_{X_w (r)}  \le c r^m \Big(\| \psi_1 \|_{X_w (r)} +\| \psi_2 \|_{X_w(r) }\Big)\\ 
&\leq  c\Big(\overline{\overline{o}}(1)  \sum_{|\alpha |=m}r^m \| D^\alpha  \varphi \|_{X_w(r)}  +\sum_{|\alpha|<m} r^{m-|\alpha|} r^{|\alpha |}
 \| D^\alpha  \varphi \|_{X_w(r)}  \Big) \\ 
&\leq c \Big(\overline{\overline{o}}(1) \sum_{|\alpha|=m}r^m \| D^\alpha  \varphi \|_{X_w(r) }  +r\sum_{|\alpha |<m}r^{|\alpha| }
 \| D^{\alpha } \varphi    \|_{X_w(r)}  \Big)\\
&\leq   \sigma(r) \| \varphi \|_{W_{X_w ;r}^m (r)}.
\end{align*}  
Taking into account   \eqref{eq-3-1} we obtain 
$$
\| \cT_0\varphi \|_{W_{X_w ;r}^m (r) }  \leq  \sigma(r) \| \varphi \|_{W_{X_w ;r}^m (r) } , \qquad   \sigma(r) \to 0 \  \text{ as }  \  r \to 0.
$$

\end{proof}

 \section{Local existence of strong  solutions}\label{sec4}

In the next section we are going to obtain some local results concerning  solvability in  weighted Sobolev-Banach space builded upon  $X_w.$

\begin{lemma}\label{lem-4-2}
 Let   \eqref{el-cond}, \eqref{eq-P}, the Property~\ref{Hyp-2-1}  and  the condition  \eqref{E} hold true in a ball $\B_r (x_0)\Subset \Omega.$  Then
\begin{enumerate}
\item 
If  $\varphi \in W^m_{X_w} (\B_r (x_0))$,  then  $\varphi=\cT_{x_0 } \varphi +\cS_{x_0 } \cL\varphi. $
\item
  If for some $f\in X_w (\B_r (x_0 ))$  it  holds $\varphi =\cT_{x_0 } \varphi +\cS_{x_0 } f$, then $\varphi $ is a  strong solution of the  equation $\cL\varphi(x) =f(x)$ for almost all $x\in \B_r(x_0).$ 
\end{enumerate}
\end{lemma}
\begin{proof}
  \mbox{(1)} Without loss of generality we may assume that $x_{0} =0\in \Omega. $ 
        Because of  \eqref{eq-spaces}
 there exists  some $ p_0 >1$ such that  
 $W_{X_w }^m (\Omega )\subset W^{p_0,m} (\Omega )$  continuously.

Since 
$\varphi \in W^{p_0,m} (r)$   than  we can  consider  $\varphi \in \overset{\circ \qquad  } {W^{m_,p_0}} (r+\delta)$ for some $\delta>0$
 small enough   
and  
  $\varphi =\cT_{0} \varphi +\cS_{0} \cL\varphi, $ (cf. \cite[Lemma A]{3}), i.e. $\Id=\cT_{0} +\cS_{0} \cL$, where $\Id$ is the
  {\it identity operator}  in $W^{p_0,m}(r). $   
  Then
 \begin{equation}\label{T0phi}
\begin{split}
\cT_0 \varphi =\cS_{0} (\cL-L_0 ) \varphi =\cS_0 \psi =\int_{\B(r+\delta)}J_{0} (x-y)\psi(y)\,dy,
\end{split}
 \end{equation}
 where $\psi= (\cL-L_0 ) \varphi  $   and   $ x\in \B_{r+\delta } . $
It is clear  that $\psi \in L^{p_0 } (r+\delta)$ and from the classical theory (cf. \cite{3,GT})  we have the representation formula 
\begin{equation} \label{eq-4-1} 
D^{\beta} \cT_0\varphi(x)  =\int_{\B_{r+\delta} } D_x^\beta J_0 (x-y)\psi(y) \,dy+ C \psi(x)                         ,
\end{equation} 
for each $|\beta|=m$ and  $C$ being a positive constant independent of $\psi.$  
                       
 Since $D_x^m J_{0} $ is a Calder\'on-Zygmund type  kernel, then from the boundedness of the  singular operators
 in $L^{p_0 } (r+\delta )$  it follows that the formula \eqref{eq-4-1} holds also  for 
 $\psi \in L^{p_0 }(r+\delta ).$
By the Property~\ref{Hyp-2-1}  and \eqref{eq-operators}  it follows that \eqref{eq-4-1} is  valid also in  $X_w (r+\delta )$  and
\begin{equation} \label{eq-4-2} 
\begin{split}
\| D^\beta \cT_0\varphi\|_{X_w(r) }& \le \| D^\beta \cT_0\varphi \|_{X_w(r+\delta )} \le c \| \psi \|_{X_w (r+\delta )}\\
& \le 
c\Big(\| L_{0} \varphi \|_{X_w(r+\delta )} +\| \cL\varphi \|_{X_w(r+\delta )} \Big).
\end{split}
\end{equation} 
  By the assumptions on the coefficients  we have $|a_{\alpha}(0)|\le \| a_{\alpha} \|_{L^{\infty } (r+\delta )} $ for all   $|\alpha |=m$. Then from \eqref{eq-4-2} it follows 
$$
\| D^m \cT_0\varphi \|_{X_w (r)} \leq c \| \varphi \|_{W_{X_w }^{m} (r+\delta )}.
$$
Since $W_{X_w }^{m} (\Omega )$ verifies the Property~\ref{Hyp-2-2}  we have 
$
\| D^m \cT_0\varphi \|_{X_w (r)} \le c \| \varphi \|_{W_{X_w}^m (r)}
$
and
$$
\| \cT_{0} \varphi \|_{W_{X_w }^m(r)} \le c \| \varphi \|_{W_{X_w }^{m} (r)}
$$ 
 that implies $\cT_{0} \in [W_{X_w }^{m} (r)]$. 

Analogously  $\cS_{0} \in [X_w(r); W_{X_w}^m (r)]$ and hence $\cS_{0} \cL\in [W_{X_w}^m (r)].$
 Then  
$$
 \varphi(x) =\cT_{0} \varphi(x) +\cS_{0} \cL\varphi(x) \quad \text{ for   a.a. } x\in \B_r,
$$
  in $L^{p_{0} } (r).$
 The estimate \eqref{eq-phi} holds also in  $W_{X_w}^m(r)$, i.e. 
\begin{equation}\label{eq-id}
\Id=\cT_{0} +\cS_{0} \cL \in [W_{X_w}^m (r)].
\end{equation}
 
\mbox{(2)}  
Let    $f\in X_w (r)$  such that  $\varphi =\cT_{0} \varphi +\cS_{0} f.$  Hence 
$$
\cL\varphi(x) =L_{0} \cS_{0} \cL\varphi(x) =L_{0} \cS_{0} f(x) = f(x)\quad \text{ for a.a. }  x\in \B_r,
$$
where we have used that
$L_{0} \cS_{0} =\Id$ in $X_w(r).$   
Then  
$
\cL\varphi(x) =f(x).
$
\end{proof}

Because of the equivalence of the norms \eqref{norm1} and \eqref{norm2} the Lemma~\ref{MainLemma} holds also in $W^m_{X_{w;r} }(r). $ 
 \begin{corollary}\label{cor-4-1}   Let $\varphi \in W^m_{X_{w;r}} (\B_r(x_0)).$  Under the conditions of   Lemma~\ref{lem-4-2} it holds
\begin{enumerate}
\item   $\varphi =\cT_{x_{0} } \varphi +\cS_{x_{0} } \cL\varphi, $  a.e. in $\B_r(x_0);$
\item  if for some  $f\in X_w (\B_{r}(x_{0} ))$ it holds $\varphi =\cT_{x_0 } \varphi +\cS_{x_0 } f$, then $\varphi $ is a solution of the equation $\cL\varphi =f.$ 
\end{enumerate}
\end{corollary}

 The previous lemma  allows  us to proof the following  local  existence result. 

\begin{theorem}\label{thm-4-1}  Let \eqref{el-cond}, \eqref{eq-P},  and the  Properties~\ref{Hyp-2-1} and~\ref{Hyp-2-2} hold.
Then  the equation $\cL u=f $  admits a solution   in $W_{X_w }^m (\B_{r} (x_{0} )),$  for all  $ f\in X_w (\Omega ),$   and $r>0$ small enough. 
\end{theorem}

 \begin{proof}   By the Corollary~\ref{cor-4-1} we can  give the   proof in the space $W^m_{X_{w;r} }(r), $ taking $x_{0} =0\in \Omega.$ 
 Since  $L_{0} \cS_{0} =\Id $ in $X_w (r),$ and keeping in mind \eqref{eq-Tx0} we have 
\begin{equation}\label{eq2}
L_{0} \cT_{0} =L_{0} \cS_{0} (L_{0} -\cL)=\Id (L_{0} -\cL)=L_{0} -\cL .
\end{equation}
 For any   $u \in W_{X_{w;r} }^{m}(r)$ we consider  the equation $\cL u=f.$  By  \eqref{eq2} we can rewrite it  as follows 
$$
f=\cL u= (L_{0} -L_{0} \cT_{0}) u=L_0(\Id -\cT_0) u ,
$$
where the identity  operator here acts in $W_{X_{w;r} }^{m}(r).$
Applying $\cS_0$ we obtain  
$$
\cS_0 f= \cS_0L_0(\Id -\cT_0) u = (\Id - \cT_0) u. 
$$ 
 By   Lemma~\ref{MainLemma} we have  
$$
\| \cT_{0} \| _{[W_{X_w ;r}^{m} (r)]} =\sigma(r) \to 0,\qquad \text{ as } r\to 0.
$$
 Taking $r$ small enough such that  $\| \cT_{0} \| _{[W_{X_w ;r}^{m} (r)]} <1,$ we obtain that  the operator $\Id  -\cT_{0} $ is boundedly invertible in $W_{X_w ;r}^{m} (r)$ and  by Lemma~\ref{lem-4-2} the function
  $$
u=(\Id -\cT_{0} )^{-1} \cS_{0} f,
$$ 
 is a solution of the equation $\cL u=f$   in $W_{X_w ;r}^{m} (r).$

\end{proof}

\section{Interior Schauder type estimates}\label{sec5}

 Our goal  now is to obtain local  interior Schauder type estimates for the solutions of  $\cL \varphi=f$ in 
 $W_{X_w}^m (\Omega).$   For this purpose  we need some auxiliary lemma.

 Let $\omega\in C_0^\infty([0,1]$ be  such that 
$$
\omega (t)=\begin{cases}
1 &\quad 0\le t<\frac{1}{3}, \\ 
0 &\quad  \frac{2}{3} <t\le 1 .
\end{cases} 
$$ 
Then we can define a cut-off function $\xi\in C_0^\infty(\B_{r_2})$ as
\begin{equation} \label{eq-5-1} 
\xi (x)=
\begin{cases} 1&\quad |x|\le r_{1}, \\
\omega\Big(\frac{|x|-r_{1} }{r_{2} -r_{1} } \Big) &\quad  r_{1} <|x|\le r_{2},
\end{cases}  
\end{equation} 
for any   $0<r_1 <r_2\leq 1 .$   The norm of $\xi$ is bounded,  as it is proved in  \cite{9}
 $$
\| \xi \|_{C^m (r_2)} \le c\Big(1-\frac{r_{1} }{r_{2} } \Big)^{-m} 
$$ 
with a constant independent of $r_1$ and $r_2.$

\begin{lemma}\label{lem-5-2}
Let the conditions of Theorem~\ref{thm-4-1} be   fulfilled in $\B_{r_2}\Subset \Omega. $   Then  for any $0<r_1<r_2$ as in \eqref{eq-5-1}  and  $ u\in W_{X_w ;r_2 }^{m} (\Omega )$ the following estimate holds 
$$
\| u\|_{W_{X_w ;r_1 }^m(r_1) } \le c\Big(1-\frac{r_{1} }{r_{2} } \Big)^{-m}\Big (\| \cL u\| _{X_w (r_2 )} +
\| u\|_{W_{X_w ;r_2 }^{m-1} (r_2 )}\Big)  
$$ 
 with a constant   independent  of   $r_1, r_2,$ and $u.$
\end{lemma}

\begin{proof} 
  Take  $\varphi =\xi  u \in W_{X_w ;r_2 }^m (r_2 )$    with a compact support in  $\B_{r_2 }. $  Then by Corollary~\ref{cor-4-1}  we have 
\begin{equation} \label{eq-5-2} 
\varphi =\cT_{0} \varphi +\cS_{0} \cL \varphi           \end{equation} 
and by Lemma~\ref{MainLemma} there exists $ r>0$ such  small that  
$$
\| \cT_{0} \varphi \|_{W_{X_w ;r_2 }^m (r_2 )} \le \frac{1}{2} \| \varphi \|_{W_{X_w ;r_2 }^m (r_2 )} 
$$ 
holds for all  $r_2 \in (0,r).$   Then by  \eqref{eq-5-2} we obtain 
\begin{equation} \label{eq-5-3} 
\| \varphi \|_{W_{X_w ;r_2 }^m(r_{2} )} \le 2\| \cS_{0} \cL \varphi \|_{W_{X_w ;r_2 }^{m} (r_2 )}                    \end{equation} 
where 
\begin{equation}\label{eq-S0}
\cS_{0} \cL\varphi(x) =\int _{\B_{r_2 } }J_{0} (x-y)\cL\varphi (y)\,dy.
\end{equation}
Calculating  the higher order  derivatives we obtain
$$
D^\beta \cS_{0} \cL\varphi(x) =\int_{\B_{r_2 } } D_x^\beta J_{0} (x-y)\cL\varphi (y)\,dy +c\cL\varphi (x),\quad |\beta|=m
$$ 
{with a constant $c$ independent of $\varphi.$
The integral operator   is  of Calder\'on-Zygmund type and, by  Property~\ref{Hyp-2-1},    the following  estimate holds}
\begin{equation} \label{eq-5-4} 
\| D^m \cS_{0} \cL\varphi \|_{X_w (r_2)} \le c\| \cL\varphi\|_{X_w (r_2)}  .      \end{equation} 
For the lower order derivatives of \eqref{eq-S0} we have the following expression
$$
D^\beta \cS_{0} \cL\varphi(x) =\int _{\B_{r_2 } }D_x^\beta J_{0} (x-y)\cL\varphi (y)\,dy,\quad |\beta|<m . 
$$ 
In this case the  kernel has a weak singularity and the integral operator is a Riesz type integral  \eqref{eq-R} that  we can estimate   as 
\begin{equation} \label{eq-5-5} 
r_2^{|\beta|} \| D^\beta \cS_{0} \cL\varphi \|_{X_w(r_2)} \le c r_2^m \| \cL\varphi \|_{X_w (r_2)}                            
\end{equation} 
with a  constant independent of $r_2$ and $\varphi. $  Unifying  \eqref{eq-5-4} and \eqref{eq-5-5}   we obtain 
\begin{equation} \label{eq-5-6} 
\| \cS_{0} \cL\varphi \| _{W_{X_w ;r_2 }^m (r_2 )} \le c r_2^m \| \cL\varphi \| _{X_w(r_2 )}  .                                    
\end{equation} 
On the other  hand, it is easy to see that $\cL\varphi $ can be represented in the form 
\begin{equation} \label{eq-5-7} 
\cL\varphi(x) =\xi(x)  \cL u(x)+M(u;\xi ) ,                                              
\end{equation} 
where $M(u;\xi )$ is a linear combination of derivatives $D^\alpha u$, the order of which does not exceed $(m-1)$, multiplied by the derivatives of $\xi $ of order at most $m.$  Precisely 
$$
M (u;\xi )=\sum_{\overset{|\alpha|+|\alpha'|\leq m}{ |\alpha|<m}} c_{\alpha}(x) D^{\alpha'} \xi(x) D^\alpha u(x).
$$ 
Following   \cite{9} we obtain analogously 
\begin{equation}\label{eq3}
r_2^m \| M(u;\xi ) \|_{X_w (r_2)} \le c \| \xi \|_{C^m (r_2 )} \| u\|_{W_{X_w ;r_2 }^{m-1} (r_2 )}  .
\end{equation}
Unifying \eqref{eq3} and  \eqref{eq-5-6}, we obtain
\begin{equation} \label{eq-5-8}
\begin{split}
\| \cS_{0} \cL u &\|_{W_{X_w ;r_1 }^m (r_1 )} \leq \| \cS_{0} \cL\varphi \| _{W_{X_w ;r_2 }^m (r_2 )}\\
 &\le c r_2^m \Big(\|\xi\|_{C^m(r_2)}\|\cL u\|_{X_w(r_2)} + 
\|M(u,\xi)\|_{X_w(r_2)}\Big)\\
&\leq  c \|\xi\|_{C^m(r_2)} \Big(\|\cL u\|_{X_w(r_2)} +  \| u\|_{W_{X_w ;r_2 }^{m-1} (r_2)}\Big)\\
&\leq c  \Big(1-\frac{r_{1} }{r_{2} } \Big)^{-m}    \Big( \|\cL u\|_{X_w(r_2)} +  \| u\|_{W_{X_w;r_2 }^{m-1} (r_2)}\Big).
\end{split}
\end{equation}
Then making use of \eqref{eq-5-3} we  obtain the desired estimate
\begin{align*}
\|  u \|_{W_{X_w ;r_1 }^m (r_1 )}& = \|  \varphi \|_{W_{X_w ;r_1 }^m (r_1 )} \leq 2 \|\cS_0 \cL\varphi\|_{W_{X_w ;r_1 }^m (r_1 )}\\[6pt]
&\leq  c  \Big(1-\frac{r_{1} }{r_{2} } \Big)^{-m}    \Big( \|\cL u\|_{X_w(r_2)} +  \| u\|_{W_{X_w;r_2 }^{m-1} (r_2)}\Big).
\end{align*}

\end{proof}

In order to establish interior Schauder's estimate we need the following result.
\begin{lemma}\label{lem-5-3}
Suppose that  the space $X_w(\Omega)$ possesses the Properties~\ref{Hyp-2-1},{}~\ref{Hyp-2-2} and~\ref{Hyp-2-3}.  Then
$$
\| u\|_{W_{X_w }^{k} (\Omega )} \le \varepsilon \| u\|_{W_{X_w }^{k+1} (\Omega)} +\frac{c}{\varepsilon^{p^k } } \| u\|_{X_w (\Omega )} 
$$ 
for some $p>1$ and for all $ k=1,\ldots, m-1.$ 
\end{lemma}
 \begin{proof} 
Without loss of generality we may  assume that $d_{\Omega } =1$ (see  \cite{9}).  Let $\Omega \Subset \Omega'$,  where
 $\Omega'$ is a    bounded domain with sufficiently smooth boundary.
 %Denote by $\overset{\circ}{W}{}_{X_w }^m(\Omega')$ the  $W_{X_w }^m(\Omega' )$-functions, compactly supported in $\Omega'.$

In order to  simplify the calculus, we  begin assuming  \mbox{$n=1, m=2$} and then we  extend the result via induction to more general situation.
 Let $\Omega =(a,b)$ be  an interval of length $b-a=\varepsilon $ for a fixed  $\varepsilon$ and suppose that   $u\in C_0^2 (\Omega'),$ then by  
   \cite[Theorem~7.27]{GT} we have  
\begin{equation} \label{eq-5-9} 
|u'(x)|\le \int _{a}^{b}|u''(t)|\,dt+\frac{18}{\varepsilon^{2} } \int _{a}^{b}|u(t)|\,dt, \quad   \forall \, x\in (a,b).
\end{equation}
 By density arguments    it follows that  \eqref{eq-5-9}
holds true also for any   $ u\in \overset{\circ}{W}{}^{2,1} (\Omega').$   Since   $\overset{\circ}{W}{}_{X_w }^2  (\Omega')\subset\overset{\circ }{W}{}^{2,1} (\Omega')$, then it is evident that this inequality is true  for all  
$ u\in \overset{\circ}{W}{}_{X_w}^2  (\Omega').$ 

We can estimate the norm of $u'w$ by \eqref{eq-5-9}:
\begin{equation} \label{eq-5-10} 
\begin{split}
\| u'w\|_{X(a,b)} 
\leq      \| w\|_{X(a,b)}   \Big(\int_a^b |u''(t)|\,dt+\frac{18}{\varepsilon^2 } \int_a^b |u(t)|\,dt \Big).
\end{split}
\end{equation} 
Then for an  arbitrary interval $\Omega=(a,b)$ we construct covering   with  disjoint intervals $I_l$ with length $\varepsilon, $ such that  $
\Omega =\bigcup_l I_l.
$
Since  \eqref{eq-5-10} holds     for all  $ I_l ,$  the H\"older inequality for $p>1$ gives 
\begin{equation}\label{eq-we}
\begin{split}
\| u'w\| _{X(I_l )}^p \le & \|w\|_{X(I_l )}^p \Big[\Big( \int_{I_l }|u''(t)|\, dt\Big)^p +\frac{c}{\varepsilon^{2p} } \Big(\int_{I_l }|u(t)|\, dt\Big)^p  \Big] \\
\le& c \|w\|_{X(\Omega)}^p \Big(\varepsilon^{p-1}  \int _{I_l }|u''(t)|^p\,  dt+\frac{c }{\varepsilon^{p+1} } \int _{I_l }
|u(t)|^p\,  dt \Big),
\end{split}
\end{equation}
where we have used that for each $I_l\subset \Omega $ we have 
$  \| w\| _{X(I_l)}\leq \|w\|_{X(\Omega)}. 
$
   Then
$$
\sum_l \| u'w\| _{X(I_l )}^p  \le c\|w\|_{X(\Omega)}^p  \Big(\varepsilon^{p-1 } \int _{\Omega }|u''(x)|^p\,  dx+\frac{c}{\varepsilon^{p+1} } 
\int_{\Omega }|u(x)|^p\,  dx\Big).
$$ 

Let  the Condition~\ref{Con-2-1} holds. Then for each   $p\in[1,p_1]$ we have 
$$
\| u'w\|_{X(\Omega )}^p \leq c \sum_l \| u'w\|_{X(I_l )}^p .
$$
For any  $p\leq \min\{p_0,p_1\},$   the  Property~\ref{Hyp-2-1} implies   $X_w (\Omega )\subset L^p (\Omega )$  and  hence
\begin{equation}\label{eq-uw}
\begin{split}
\| u'w\|_{X(\Omega )} & \le c\|w\|_{X(\Omega)}  \Big(\varepsilon^{p-1} \int_\Omega |u''(x)|^p\, dx +\frac{c}{\varepsilon^{p+1}}  
\int_{\Omega }|u(x)|^p\, dx \Big)^{\frac{1}{p} }\\
&  \le c\|w\|_{X(\Omega)} \Big( \varepsilon \| u''\|_{L^p (\Omega )} +\frac{c}{\varepsilon^{\frac{p+1}{p-1}}}
\| u\|_{L^p (\Omega )} \Big).
\end{split}
\end{equation} 

If   $x\in \R^n$ then \eqref{eq-uw} holds for any partial derivative $D_iu$ for $i=1,\ldots,n.$ Moreover it holds  also for the higher order derivatives, that is 
$$
\| D^k u\|_{X_w (\Omega )} \le \|w\|_{X(\Omega)}   \Big( \varepsilon^{\frac{1}{p'} }\| D^{k+1 } u\|_{X_w (\Omega)} +\frac{c}{\varepsilon^{2-\frac{1}{p'} }}
 \| D^{k-1 } u\|_{X_w (\Omega )} \Big),
$$ 
with $ 1\leq k\leq m-1.$ Summing up along $k$   we obtain  
$$
\| u\| _{W_{X_w (\Omega )}^s } \le \|w\|_{X_w} \Big(\epsilon\| u\|_{W_{X_w (\Omega )}^{s+1} } +\frac{c}{\epsilon^{2-\frac1{p'}} } \| u\|_{W_{X_w }^{s-1} (\Omega )} \Big)
$$ 
for all $s=1,\ldots,m-1.$

Consider  now the case when Condition~\ref{Con-2-2} holds, hence there exists $p_2\in(1,\infty) $ such that   
$$
\| u'\|_{X_w (\Omega )} \le c\| u'\|_{L^{p} (\Omega )}\qquad \forall \ p\in [p_2,\infty).
$$ 
As before, it is sufficient to consider the case of $n=1,$ $m=2$ and $b-a=\varepsilon.$ 
Taking into account  \eqref{eq-5-9} and \cite[Lemma 7.27]{GT} (see also \cite{APS,26}) we have 
\begin{equation} \label{eq-5-11} 
\begin{split}
\| u'\|_{X_w (a,b)}  & \leq    c\Big( \int_a^b |u'(x)|^p\, dx\Big)^{\frac1p}\leq 
c \varepsilon\int_a^b |u''(t)|\, dt+\frac{c}{\varepsilon^{2p}} \int_a^b |u(t)|\, dt
\end{split}
\end{equation}

Taking an  arbitrary interval $\Omega =(a,b)$ we  construct  a covering with a family of disjoint intervals of length  $\varepsilon,$ that is,     $\Omega =\bigcup_{l}I_l  $,    where $\Omega \bigcap I_l \ne \emptyset,   $    $I_{i} \bigcap I_{j} =\emptyset ,$  $i\ne j$.  Since  \eqref{eq-5-11} holds for all $I_l,$   summing up these estimates   with respect to $l$   we obtain 
\begin{equation*}\label{eq-5-11a}
\begin{split}
\|u'&\|_{X_w(\Omega)} \le \varepsilon \|u''\|_{L^p(\Omega)} 
+\frac{c}{\varepsilon } \|u\|_{L^p(\Omega)}\\ 
& \ \le \epsilon \int_{\Omega }|u''(x)|w(x)w^{-1}(x) \,dx
 +\frac{c}{\epsilon^{2p_2 } } 
\int_{\Omega }|u(x)|w(x)w^{-1}(x)\, dx \Big)\\
& \  \le \| w^{-1} \|_{X'(\Omega )} \Big(\epsilon\| u''\|_{X_w (\Omega )}
 +\frac{c}{\epsilon^{2p_2 } } \| u\|_{X_w (\Omega)} \int_a^b |u''(t)|\, dt +\frac{c}{\varepsilon^2}\int_a^b |u(t)\,dt|\Big)
\end{split}
\end{equation*}
 Extending this  to functions of $n $ variables  and taking  $\varepsilon =  \| w^{-1} \|_{X'(\Omega )} \epsilon,$ we obtain
\begin{equation}\label{eq-5-11b}
\| u\|_{W_{X_w (\Omega )}^k } \le \varepsilon\| u\|_{W_{X_w(\Omega )}^{k+1} } +\frac{c}{\varepsilon^p } \| u\|_{W_{X_w }^{k-1} (\Omega )}, \quad    \forall  \ k=1,\ldots, m-1, 
\end{equation} 
for any $p\geq 2p_2$ and a  constant $c$ depending on known quantities  and on the norms  $\|w^{-1}\|_{X'(\Omega)}$ and 
$\|w\|_{X_w(\Omega)}.$

Introducing the notion 
 $A_k :=\| u\|_{W_{X_w (\Omega )}^{k} } ,$ for all $ k=0,\ldots,m,$  we rewrite \eqref{eq-5-11b}  
\begin{align*}
A_{1} &\le \varepsilon_1 A_2 +\frac{c_1}{\varepsilon_1^p } A_{0}\\
A_2 &   \le  \varepsilon_2 A_3 +\frac{c_2}{\varepsilon_2^p } A_1  \leq
\varepsilon_2 A_3 +\frac{c_1c_2}{(\varepsilon_2 \varepsilon_1)^p } A_{0} +
\frac{c_2 \varepsilon_1 }{\varepsilon_2^p } A_2 .
\end{align*}
Taking  $\varepsilon_2=\varepsilon, \varepsilon_1= \varepsilon^{p+1}$ with $\varepsilon c_2<1,$  we get
$$
A_{2} \le \varepsilon A_{3} +\frac{1}{\varepsilon^{p^2} }  A_{0} .
$$
Repeating the same procedure we obtain a kind of  interpolation inequality
$$
A_{k} \le \varepsilon A_{k+1} +\frac{1}{\varepsilon ^{p^{k} } } A_{0}  ,\qquad  \forall\,  k=1,\ldots,m-1
$$ 
that is the assertion of the lemma.
\end{proof}

\begin{theorem}[Interior Schauder type estimate]\label{thm-5-1} 
Consider the  uniformly elliptic  equation $\cL u=f$  in a bounded domain $\Omega$  with  principal coefficients
$a_{\alpha}\in C(\overline{\Omega }),$ for all $|\alpha|=m$ and $ a_{\alpha}\in L^\infty(\Omega) $ for $|\alpha|<m.$ 
Suppose that   the Properties~\ref{Hyp-2-1} and~\ref{Hyp-2-2} hold.  Then for any domain $\Omega_0 \Subset \Omega $ the  following a priori estimate holds
\begin{equation}\label{eq-th5.3}
\| u\|_{W_{X_w }^m (\Omega_0)} \le C (\| \cL u\|_{X_w (\Omega )} +\| u\|_{X_w (\Omega)} ).
\end{equation}
\end{theorem}

\begin{proof}  We can construct a finite cover of  $\Omega_0$ with  balls whose closure belong to $\Omega.$ That is way, without loss of generality we can  consider  $\Omega $ and $\Omega _{0} $ to be  concentric balls of small radius centered at  zero.  Let $R>0$  be sufficiently small, then we need  to prove the following estimate
$$
\| u\|_{W_{X_w ;r}^{m} (r)} \le C \Big(1-\frac{r}{R} \Big)^{-m^2 } (\| \cL u\|_{X_w (R)} +\| u\|_{X_w (R)} ).
$$ 
 Set
$$
A=\sup_{0\le r\le R}  \Big(1-\frac{r}{R} \Big)^{m^{2} } \| u\|_{W_{X_w ;r}^m (r)}.
$$ 
 Suppose that $\| u\|_{W_{X_w ;R}^m (R)} \ne 0$, otherwise, there is nothing to prove. It is easy to see that there  exists  $ r_{1} \in (0,R)$, such that 
$$
A\le 2\Big(1-\frac{r_1 }{R} \Big)^{m^{2} } \| u\|_{W_{X_w ;r_{1} }^{m} (r_{1} )} .
$$ 
For a fixed $r_{2} \in (r_{1} ,R)$ and  Lemmata~\ref{lem-5-2}  and~\ref{lem-5-3} we obtain 
\begin{align*}
A\le & \, c \Big(1-\frac{r_{1} }{R} \Big)^{m^2 } \Big(1-\frac{r_{1} }{r_{2} } \Big)^{-m} (\|\cL u\|_{X_w(r_2)} +
\| u\| _{W_{X_w ;r_2 }^{m-1}(r_2 )} )\\
\le & \, c \Big(1-\frac{r_{1} }{R} \Big)^{m^{2} } \Big(1-\frac{r_{1} }{r_{2} } \Big)^{-m} \Big(\|\cL u\|_{X_w (r_{2} )} + 
 \varepsilon \| u\|_{W_{X_w ;r_2 }^m (r_2 )}  +\frac{c}{\varepsilon^{p^{m-1} } } \| u\|_{X_w (r_2 )}\Big ).
\end{align*} 
Keeping in mind  that 
$
\Big(1-\frac{r_{2} }{R} \Big)^{m^{2} }\| u\|_{W_{X_w ;r_2 }^{m} (r_2)} \le A,
$
we obtain
\begin{equation}  \label{eq-5-12} 
\begin{split}
A\le  c&\Big(1-\frac{r_1 }{R} \Big)^{m^2 } \Big(1-\frac{r_1 }{r_2 } \Big)^{-m}\| \cL u\|_{X_w (r_2)}\\
 &+ \varepsilon c \Big(1-\frac{r_1 }{R} \Big)^{m^2 } \Big(1-\frac{r_1 }{r_2 } \Big)^{-m} \Big(1-\frac{r_2 }{R} \Big)^{-m^2 }A\\
& + \frac{c}{\varepsilon^{p^{m-1}}} \Big(1-\frac{r_1 }{R} \Big)^{m^2}  \Big(1-\frac{r_1 }{r_2 } \Big)^{-m} \| u\|_{X_w (r_2)} .     
\end{split}                               
\end{equation} 
 Set $\delta =1-\frac{r_1}{R}$ and choose $r_2$ and $\varepsilon $ in the following way
$$
1-\frac{r_2}{R}=\frac{\delta}{2}, \quad \varepsilon= 2^{-1-4m^2}\delta^{2m}c^{-1}.
$$
This implies 
$$
0<\delta<1, \quad \frac{\delta}2<1-\frac{r_1}{r_2}<\delta,\quad 0<\varepsilon <1
$$
and, hence  
$$
\varepsilon c\Big(1-\frac{r_{1} }{R} \Big)^{m^{2} } \Big(1-\frac{r_{1} }{r_{2} } \Big)^{-m} \Big(1-\frac{r_{2} }{R}\Big)^{-m^{2} } <\frac{1}{2} .
$$ 
Making use of  \eqref{eq-5-12} we obtain the estimate 
$$
A< c \big(\| \cL u\|_{X_w (r_2 )} +\| u\|_{X_w (r_2)} )\le c\Big(\| \cL u\|_{X_w (R)} +\| u\|_{X_w (R)}\Big),
$$
that gives   \eqref{eq-th5.3}.
\end{proof}

\section{Some examples  of  Banach Function Spaces}

 Let  $X(\Omega )$ be  correctly defined for any domain  $\Omega \subset \R^{n} $ with a sufficiently smooth boundary   $\partial \Omega, $ and suppose that  $\Omega _0$ is a bounded domain such that  
 $\Omega \subset \Omega_0\subset \R^{n} .$

\begin{lemma}\label{lem-6-1}
Let  $ \theta_k \in [W_{X_w }^k (\Omega ); W_{X_w }^{k} (\Omega _0)]$ be such that $ \theta_k f\big\vert_{\Omega } =f $ for all $ k=0,\ldots,m.$  Then there exists  $ \theta \in [W_{X_w }^{k} (\Omega );W_{X_w }^{k} (\Omega _0)] ,$ such that $ \theta f\big\vert_{\Omega } =f, $ for all $k=0,...,m$. 
\end{lemma}
\begin{proof}
Consider the following operator
$$
\theta  f=
\begin{cases} 
\theta _{m} f & f\in W_{X_w}^m (\Omega) , \\
\theta_{m-1} f  & f\in W_{X_w }^{m-1} (\Omega )\backslash W_{X_w}^m (\Omega ) , \\
\ldots &\ldots\ldots\ldots\ldots \\
\theta_1 f  & f\in W_{X_w }^1 (\Omega )\backslash W_{X_w }^{2} (\Omega ) , \\
\theta_{0} f &  f\in  X_w (\Omega )\backslash W_{X_w }^{1} (\Omega) .
 \end{cases}
$$

  From the chain of embedding  
$$
X_w (\Omega )\supset W_{X_w }^1 (\Omega )\supset \ldots \supset W_{X_w }^m (\Omega ),
$$
it  follows that the linear operator $\theta \in [X_w  (\Omega )]$ is well defined.

 Consider the case $m=1,$  and extend the proof by induction to $m>1.$ For simplicity we carry out the proof in the one dimensional case.
 Supposing 
$$
\theta \in [W_{X_w }^{1} (\Omega ); W_{X_w }^{1} (\Omega _0)] \quad \text{ and } \quad \theta  f\big\vert_{\Omega } =f,
$$ 
we have to show that $\theta \in[X_w (\Omega ); X_w (\Omega _0)].$  Indeed, if 
$f\in  X_w (\Omega ) \setminus  W_{X_w }^1 (\Omega ),$  then
 $\theta  f=\theta_0 f$ and therefore 
$$
\| \theta  f\|_{X_w (\Omega _0)} \le \| \theta_0 \| \| f\|_{X_w (\Omega)}.
$$
 
 Let $f\in W_{X_w }^{1}(\Omega ),$ hence $\theta  f=\theta _1 f $. 
Consider the function
$$
F(x)=\int_{a}^{x}\theta _{1} f(t)\,dt  \qquad\text{ for a.a. } \  x\in \Omega _0,
$$
 where $a\in \Omega $ is some fixed point.
 Then 
$
F\in W_{X_w}^{1}  (\Omega _0),$      $ \frac{dF(x)}{dx} =\theta _{1} f(x),$ for a.a. $ x \in \Omega_0$,  and 
 $\theta F=\theta _{1} F$ that implies  $F\in W_{X_w }^{1} (\Omega)$  and  
\begin{equation} \label{eq6-1} 
\| \theta_1  F\|_{W_{X_w}^1 (\Omega _0)} \le C_1 \| F\|_{W_{X_w }^{1} (\Omega )} ,                                       
\end{equation} 
where $C_1$ is a constant, independent of  $F.$  It is obvious that $\frac{dF(x)}{dx}=f(x),$  a.e. in $\Omega. $  Moreover  
\begin{align*}
\| \theta _{1} f\| _{X_w  (\Omega _0)}& \leq \| \theta _{1} F\|_{X_w (\Omega _0)} +\| \theta _{1} f\| _{X_w  (\Omega _0)}\\
&\leq  C_1 \Big(\| f\|_{X_w  (\Omega )} +\Big\| \int _{a}^{x}f(t)\, dt \Big\|_{X_w (\Omega )} \Big)\le C\| f\|_{X_w  (\Omega )}.
\end{align*}

Considering functions of $n$ variables, fixing  all variables except one  and applying the above procedure we can extend this result in the $n$-dimensional case. 
\end{proof}

\subsection{The Marcinkiewicz spaces }

{  
 The  Banach space   $M^{p,\lambda}(\Omega), p\in(1,\infty), \lambda \in(0,1),$ consisting of all 
 functions  $f\in L^p_\loc(\Omega)$  
 for which  }
$$
\| f\|_{M^{p,\lambda}(\Omega) } =\sup_{E\subset \Omega } \Big(\frac{1}{|E|^{\lambda }} \int_{E}|f(x)|^{p} \,dx \Big)^{\frac{1}{p} }  <\infty
$$
{is  called {\it the Marcinkiewicz space} (see for instance  \cite{38}).} 

{Let us note that in the definition of 
  {\it  the Morrey spaces}  $L^{p,\lambda } (\Omega)$  the supremum is taken over all intersections  $E=\B\cap \Omega $ where $\B$ are balls in $\R^n.$  This   gives    
  the continuous embedding $M^{p,\lambda} (\Omega)\subset L^{p,\lambda } (\Omega)$   for all $\lambda\in (0,1).$ }
  
It is easy to see that   $M^{p,\lambda} (\Omega)$  is nonseparable,  
 moreover, it is a rearrangement invariant space (cf. \cite{40}). 
 
 Let $T$ be a quasi-linear operator  of
 $(p,p;q,q)$-compatible weak  type, (see \cite{2} for the definition).  Then we have the following results (cf. \cite{2,40}).
\begin{lemma}\label{lem-6-4}
{The Boyd indices  of the Marcinkiewicz space $M^{p,\lambda } (\Omega) $   are 
$$
\alpha_{X} =\beta_{X} =\frac{1-\lambda }{p}, \qquad p\in(1,\infty),  \  \lambda\in (0,1).
$$}
\end{lemma}
\begin{theorem}\label{thm-6-9}
  Let   $X$ be a rearrangement invariant space on an nonatomic complete $\sigma $-ﬁnite space with infinite  measure. Then any linear (quasilinear) operator  $T$ of  $(p,p;q,q)$-compatible weak type  is bounded in $X$; i.e.,  $T\in X$ if and only if the Boyd indices  $\alpha _{X} $ and $\beta _{X} $ satisfy the inequality $$
\frac{1}{q} <\alpha _{X} \le \beta _{X} <\frac{1}{p} , \qquad 1\leq p<q\leq +\infty.
$$
\end{theorem}
The Theorem~\ref{thm-6-9} (cf. \cite{2}) implies the validity of   the Property\ref{Hyp-2-1}   for   $M^{q,\lambda } (\Omega)$ for any  $q\in(1,\infty),$  $ \lambda \in(0,1).$ The validity of Property~\ref{Con-2-2}   can be  proved  analogously as in the   case of the  Morrey spaces. Hence, by Theorem~\ref{thm-4-1} we have the following result.
\begin{corollary}\label{cor-6-5}
Under the conditions of Theorem~\ref{thm-6-9}  there exists  $u\in  W^m_{M^{q,\lambda}}(\Omega ),$ a solution    of  $\cL u=f,$  
 for all  $ f\in M^{q,\lambda } (\Omega ),$  $q\in (1,\infty) ,$   $ \lambda\in(0,1).$
\end{corollary}
In addition, by  Theorem~\ref{thm-5-1} we obtain.
\begin{corollary}\label{cor-6-6}
Under the conditions of  Theorem~\ref{thm-6-9}  the following   a priori estimate holds
$$
\| u\|_{W_{M^{q,\lambda } }^m(\Omega_{0} )} \le c(\| \cL u\|_{M^{q,\lambda } (\Omega )} +\| u\|_{M^{q,\lambda } (\Omega )} ).
$$
\end{corollary}

\subsection{Grand Lebesgue spaces } 

 The Grand Lebesgue spaces 
$L_{q)} (\Omega ),$  $q\in(1,\infty),$ are  non separable Banach  spaces endowed by  the norm 
$$
\| f\|_{q)} = \sup_{0<\varepsilon <q-1} \Big(\varepsilon \int_{\Omega }|f(x)|^{q-\varepsilon }\, dx \Big)^\frac{1}{q-\varepsilon } .
$$
 The Property~\ref{Hyp-2-1} holds for  $L_{q)}(\Omega)$ as it is proved in \cite{22}  and  the following continuous embedding is valid
$$
L^{q} \subset L^{q)} \subset L^{q-\varepsilon } ,\qquad  \forall\,  \varepsilon \in (0,q-1).
$$ 
Consider  the Grand  Sobolev spaces  $W_{L^{q)}}^m (\Omega )$ builded on the spaces $L^{q)} (\Omega ).$  
\begin{corollary}\label{cor-6-8}
Suppose that the conditions of Theorem~\ref{thm-5-1} hold true and $x_0\in \Omega$ verifies the property  \eqref{eq-P}.  Then  there exists a solution $u\in 
W_{L^{q)}}^{m} (\B_{r} (x_{0} )) $ of the equation $\cL u=f$, for   $r$ small enough and for all  $ f\in L^{q)} (\B_r(x_0) ).$
\end{corollary}
\begin{corollary}
Under the conditions of Theorem~\ref{thm-5-1}  the following a priori estimate holds
$$
\| u\|_{W_{L^{q)}}^{m} (\Omega _{0} )} \le c\big(\| \cL u\| _{L^{q)} (\Omega)} +\| u \| _{L^{q)} (\Omega )} \big).
$$ 
\end{corollary}

\subsection{Variable Lebesgue spaces }

 {Given a bounded domain  $\Omega$ and a Lebesgue measurable function $p(\cdot):\Omega \to [1,+\infty ).$ We say that $p(\cdot)$  is {\it locally 
 log-H\"older continuous function}  if there exists a constant $c_0$ such that for all   $ x,y\in \Omega$ verifying $  |x-y|<1/2,$ it holds
$$
|p(x)-p(y)|\leq \frac{c_0}{-\log (|x-y|)} .
$$ 
We denote this function set as $LH_0(\Omega).$  It follows immediately (see \cite{15}) that if $p(\cdot)\in LH_0(\Omega)$ than it is uniformly continuous and $p(\cdot)\in L^\infty(\Omega).$}
 
{ For a Lebesgue measurable function $f\in F(\Omega)$  we define the {\it modular} associated with $p(\cdot)$ by 
$$
I_{{p(\cdot),\Omega}} (f )=\int _{\Omega }|f(x)|^{p(x)} \,dx.
$$ 
Then the {\it Variable Lebesgue spaces} $L^{p(\cdot)}(\Omega),$ $p(\cdot)\in LH_0(\Omega)$ consist  of all $f\in F(\Omega)$ such that 
$
 I_{p(\cdot ),\Omega} (f)<+\infty,$ 
for which the following  norm is finite
\begin{equation}\label{eq-Lpx}
\| f \|_{L^{p(\cdot )}(\Omega)} =\inf \Big\{\lambda >0: \  I_{p(\cdot ),\Omega} \Big(\frac{f}{\lambda } \Big)\le 1\Big\}.
\end{equation}
The spaces $L^{p(\cdot)}(\Omega),$ $p(\cdot)\in LH_0(\Omega),$ endowed with 
the norm \eqref{eq-Lpx} are  BFSs.  }

 The corresponding Sobolev spaces are  denoted by $W^{p(\cdot),m} (\Omega)$ and for a given weight  $w$  we can define also  the weighted spaces  $L^{p(\cdot )}_w (\Omega)$  and $ W^{p(\cdot),m}_w (\Omega ).$

{The  weighted $L^{p(\cdot)} $ spaces defined by  the class of variable   Muckenhoupt  weights $A_{p(\cdot)}$ are of particular interest. }
\begin{definition}\label{def-6-2} 
We say that $w\in A_{p(\cdot)}$ if 
$$
[w]_{A_{p(\cdot )} } =\sup_Q |Q|^{-1} \| w\chi_Q \|_{p(\cdot)} \| w^{-1} \chi_Q \|_{p'(\cdot )} <+\infty ,
$$ 
where the  supremum  is taken over all cubes $Q\subset \R^n $ with sides parallel to coordinate axes and $p'(\cdot)$ is the conjugate function,  
$\frac{1}{p(x)} +\frac{1}{p'(x)} =1$  for all $x\in \R.$
\end{definition}
The properties  of the weighted  Variable Lebesgue spaces (see \cite{15,22}) ensure the validity of the Property~\ref{Hyp-2-1}.  For a bounded domain $\Omega$ we have the  embedding $L^{p_{+} } (\Omega )\subset L^{p(\cdot )} (\Omega )\subset L^{p_{-} } (\Omega )$ with a measurable function $p(\cdot)$ verifying \eqref{eq-p}.

\begin{corollary}\label{cor-6-9}
Let  the conditions of Theorem~\ref{thm-4-1} hold and suppose that  
 $w\in LH_{0} (\Omega )\cap A_{p(\cdot )} (\R^{n} ),$  $ p_{-} >1,$ then there exists a solution   $u\in W^{p(\cdot),m}_w (\B_{r} (x_{0} ))  $ of the equation $\cL u=f$  for $r$ small enough and for all  $ f\in L^{p(\cdot)}_w (\Omega).$
\end{corollary}
The validity of the Property~\ref{Con-2-2} in this case  follows by  Lemma~\ref{lem-6-1}  and  \cite{15}. 
\begin{theorem}\label{thm-6-1} 
Let $\Omega \subset \R^{n} $ be a bounded domain with $C^k$ smooth boundary and $p(\cdot)\in LH_{0} (\Omega )$ such that  $1<p_{-} \le p_{+} <+\infty $. Then for every $k\ge 1$ there exists a bounded linear extension operator  
$$
\theta _{k} \in
\big[W^{p(\cdot),k} (\Omega );W^{p(\cdot),k} (\R^{n} )\big].
$$
\end{theorem}

Applying the  Theorem~\ref{thm-5-1} to the case of weighted Variable Lebesgue spaces we are able to obtain a local a priori estimate for the solution. 
\begin{corollary}\label{cor-6-10} 
Let  $\Omega _{0} \Subset \Omega .$
Under   the conditions of  Theorem~\ref{thm-5-1}  and  Corollary~\ref{cor-6-9}
the following a priori estimate holds
\begin{equation}\label{apriori-Lpx}
\| u\|_{W^{p(\cdot),m} (\Omega_{0} )} \le c(\| \cL u\|_{L^{p(\cdot )} (\Omega )} +\| u\|_{L^{p(\cdot)} (\Omega)} ).
\end{equation}
\end{corollary}

 In the  case of Variable Lebesgue spaces without any  weight, the estimate   \eqref{apriori-Lpx} is obtained in

{\bf Acknowledgment}
{\small 
The research of B. Bilalov,  and S. Sadigova is 
 supported by the  Azerbaijan Science Foundation-Grant No: AEF-MQM-QA-1-2021-4(41)-8/02/1-m-02.

 The research of L.Softova  is partially supported by the MIUR PRIN 2022, Grant No. D53D23005580006 (Salerno Unity). L. Softova  is a member of INDAM-GNAMPA.

The authors are indebted to the referee for the valuable remarks and suggestions. 

  No potential  financial or non-financial competing interests to report by the authors. }

\end{document}